\documentclass[11pt]{amsart}

\usepackage[utf8]{inputenc}

\usepackage{amssymb}

\allowdisplaybreaks[2]

\newtheorem{theorem}{Theorem}[section]

\newtheorem{lemma}[theorem]{Lemma}
\newtheorem{corollary}[theorem]{Corollary}

\theoremstyle{definition}

\newtheorem{example}[theorem]{Example}
\numberwithin{equation}{section}

\begin{document}

\author{Matej Brešar,   Kaiming Zhao}
\address{M. B.: Faculty of Mathematics and Physics,  University of Ljubljana,  and Faculty of Natural Sciences and Mathematics, University
of Maribor, Slovenia}
\email{matej.bresar@fmf.uni-lj.si}
\address{K. Z.: Department of Mathematics, Wilfrid Laurier University, Waterloo,  Canada, and   College of Mathematics and Information Science, Hebei Normal  University, Shijiazhuang,  P. R. China }
\email{kzhao@wlu.ca}

\thanks{\emph{Mathematics Subject Classification}. 17B05, 17B40, 16R60.}
\keywords{Lie algebra, biderivation, commuting linear map, centroid.}

\thanks{The first named author was supported by ARRS Grant P1-0288. The second named author is partially supported by NSF of China (Grant No. 11271109) and NSERC}

\title{Biderivations and commuting linear maps on Lie algebras}

\begin{abstract}
Let $L$ be a Lie algebra over a field of characteristic different from $2$. If $L$ is perfect and centerless, then every skew-symmetric biderivation $\delta:L\times L\to L$  is of the form $\delta(x,y)=\gamma([x,y])$ for all $x,y\in L$, where $\gamma\in{\rm Cent}(L)$,  the centroid of $L$. Under a milder assumption that $[c,[L,L]]=\{0\}$ implies $c=0$, every commuting linear map from $L$ to $L$ lies in ${\rm Cent}(L)$. 
These two results are special cases of our main theorems which concern biderivations and commuting linear maps having their ranges in an $L$-module.
We provide a variety of examples, some of them showing the necessity of our 
assumptions and some of them showing that our results cover several results from the literature.
\end{abstract}
\maketitle

\section{Introduction}

Let $L$ be a Lie algebra over a field $F$
and let $M$ be a module over $L$. Recall that a linear map $d:L\to M$ is a {\it derivation} if 
$$d([x,y])=x\cdot d(y)-y\cdot d(x)\,\,\,\mbox{for all $x,y\in L$.}$$ 
%If $M=L$ and $x\cdot y$ is the usual product $[x,y]$, we get the standard definition of a derivation from $L$ to $L$.  
We will say that a bilinear map $\delta:L\times L\to M$ is a {\em skew-symmetric biderivation} if $\delta(x,y) = -\delta(y,x)$ for all $x,y\in L$ and 
$$\delta([x,y],z)=x\cdot \delta(y,z)-y\cdot \delta(x,z)\,\,\,\mbox{for all $x,y,z\in L$.}$$ 
That is, $x\mapsto \delta(x,z)$ is a derivation for every $z\in L$ (and hence, since $\delta$ is skew-symmetric,  $x\mapsto \delta(z,x)$ is also a derivation). Next, we will say that a linear map $f:L\to M$ is a {\em commuting linear map} if 
%we define that a linear map $f:L\to M$ is a {\em commuting linear map} if 
$x\cdot f(x) = 0$ for all $x\in L$.  This condition readily implies that $x\cdot f(y) =  - y\cdot f(x)$ for all $x,y\in L$, which shows that $\delta(x,y)= x\cdot f(y) $ is a skew-symmetric biderivation. Thus, skew-symmetric biderivations may be viewed as a generalization of commuting linear maps.

If $M=L$ and $x\cdot y=[x,y]$, the above definitions coincide with the usual ones. Generalizations involving modules, which we propose, are, on the one hand, interesting in their own right, and, on the other hand,   suit the methods that we will employ.

The study of commuting maps and (skew-symmetric) biderivations has its roots in associative ring theory \cite{Bre1993, BMM}, where it has turned out to be  influential and far-reaching -- see
\cite{Bre3} and \cite{FIbook}. An interest in studying these maps on Lie algebras has grown more recently \cite{Chen2016, LL, GLZ, Hanw, tang2016, WD1, WD2}.
We will cover most of the   results from these papers by using a general but simple approach. The following notion will be of crucial importance: we define  the {\em centroid of $M$}, and denote it by Cent$(M)$,
as the space  of all
$L$-module homomorphisms from $L$ to $M$, where $L$ is viewed as an $L$-module under the adjoint action.
Thus, a linear map $\gamma:L\to M$ belongs to Cent$(M)$ if
$$\gamma([x,y]) = x\cdot \gamma(y)\quad\mbox{for all $x,y\in L$}.$$
If  $M=L$, this is the usual centroid of $L$, which can be often computed by using known  results  (e.g., \cite{BN} and \cite{D}).
%Using some known results, e.g.,  \cite{BN} and \cite{D}, Cent$(L)$
 %It is generally not hard to compute the centroid Cent$(L)$ using  some known  results, for example, \cite{BN} and \cite{D}.
The connection between the centroid and the theme of this paper is straightforward: if $\gamma\in {\rm Cent}(M)$, then $\gamma$ is a 
commuting linear map, and, moreover, 
 the map $\delta(x,y)=\gamma([x,y])$ is a skew-symmetric biderivation. Under appropriate assumptions, we will show that there are no other commuting linear maps and skew-symmetric biderivations than these.
 To describe our results  more specifically,  
we need some further notation and conventions.
For any subset $S$ of $L$, we set
$$Z_M(S)=\{v \in M\,|\, S\cdot v=\{0\}\}.$$
 Note that $Z=Z_L(L)$ is just the center of $L$. As usual, we write $L'$ for $[L,L]$.
 Throughout the paper, we assume that all our Lie algebras are over a field  $F$ such that
$$\mbox{char$(F)\ne 2$.}$$
 This assumption will not be repeated in the statements 
of the results.

In principle, a description of skew-symmetric biderivations from $L\times L$ to $M$ implies a description of commuting linear maps from $L$ to $M$. However, we will describe commuting linear maps under milder assumptions, so we will treat them separately to some extent. 
Sect.\,2 is devoted to biderivations, and Sect.\,3 to commuting linear maps. 

 In Sect.\,2, we first obtain a general formula for  skew-symmetric biderivations from $L\times L$ to $M$ (Lemma \ref{bl}).
Then we derive our main result, Theorem \ref{mt}, stating that every biderivation arises from the centroid (as above) provided that $L$ is perfect and 
 $Z_M(L)=\{0\}$.  We also provide an algorithm for describing   skew-symmetric biderivations, and show 
that a number of known results can be deduced from Theorem \ref{mt}.
At the end of the section, we show by an example that 
restricting ourselves to biderivations that are skew-symmetric is really necessary.

The main result of Sect.\,3, Theorem \ref{c1}, states that already $Z_M(L')=\{0\}$ implies that 
 every commuting linear map from $L$ to $M$ belongs to ${\rm Cent}(M)$. 
%which generalizes Theorem 5,28 in \cite{LL}. 
An algorithm for describing commuting linear maps on $L$ is also provided and several examples  given.

%In Sect. 4, we provide an example on symmetric biderivations $L\to M$, where $M$ is an $L$-module, to show that the situation is very different from skew-symmetric case. We conclude the paper with an open question on symmetric biderivations.

We conclude the introduction by a remark on the possible meaning of our results in a wider context.
It is a fact that the description of additive commuting maps on prime rings \cite{Bre1993} eventually led to the theory of functional identities on noncommutative rings \cite{FIbook}. Linear commuting maps on Lie algebras can therefore be viewed as a  testing case for developing the theory of functional identities on Lie algebras. So far, with a partial exception \cite{LL} (where, however, only finite dimensional Lie algebras were treated), their description was known only for some special examples of Lie algebras. The fact that the present paper contains a description  for a  large class of Lie algebras, which includes all simple Lie algebras, seems promising.

\section{Biderivations}

We start with the crucial lemma. The main idea of its proof, i.e., computing  $\delta([x,y],[z,w])$ in two different ways, is well-known; it was first used in associative 
rings \cite{BMM}, and later also in Lie algebras \cite{Chen2016}. The formula \eqref{ena} from the middle of the proof is actually known.
 However, the final result that we will derive using this approach is, to the best of our knowledge, new.

\begin{lemma}\label{bl}
Let $L$ be a Lie algebra   and let $M$ be an $L$-module. If $\delta:L\times L\to M$ is a skew-symmetric biderivation, then
$$  \delta(u,[x,y]) - u\cdot\delta(x,y)\in Z_M(L'), \forall u,x,y\in L.$$
\end{lemma}

\begin{proof}
For any $x,y,z,w\in L$, we have
\begin{align*}
&\delta([x,y],[z,w])\\
 =& x\cdot \delta(y,[z,w]) - y\cdot \delta(x,[z,w])\\
=& -x\cdot \delta([z,w],y) + y\cdot \delta([z,w],x)\\
=& -x\cdot (z\cdot \delta(w,y))+x\cdot (w\cdot \delta(z,y))
+ y\cdot (z\cdot \delta(w,x))-y\cdot (w\cdot \delta(z,x)).
%\\
%=& x\cdot (z\cdot \delta(y,w))-x\cdot (w\cdot \delta(y,z))
%- y\cdot (z\cdot \delta(x,w))+y\cdot (w\cdot \delta(x,z)).
\end{align*}
On the other hand,
\begin{align*}
&\delta([x,y],[z,w])\\
=&-\delta([z,w],[x,y])\\
 =& -z\cdot \delta(w,[x,y]) + w\cdot \delta(z,[x,y])\\
 =& z\cdot \delta([x,y],w) - w\cdot \delta([x,y],z)\\
=& z\cdot (x\cdot \delta(y,w))-z\cdot (y\cdot \delta(x,w))-w\cdot (x\cdot \delta(y,z))+w\cdot (y\cdot \delta(x,z)).
\end{align*}
Comparing both relations and using the assumption that $\delta$ is skew-symmetric, we obtain
\begin{equation}\label{q}
[x,z]\cdot \delta(y,w) + [y,w]\cdot \delta(x,z) = [x,w]\cdot \delta(y,z) + [y,z]\cdot \delta(x,w)
\end{equation}
Note that $\delta(y,y)=0$ since $\delta$ is skew-symmetric and  char$(F)\ne 2$.
Writing $y$ for $w$ and $x$ for $z$ in \eqref{q} therefore gives 
$$[x,y]\cdot \delta(y,x) + [y,x]\cdot \delta(x,y) =0,$$
that is,
 $2[x,y]\cdot \delta(x,y)=0$. Consequently, $$[x,y]\cdot \delta(x,y)=0.$$
A linearization on $x$ yields
$$[x,y]\delta(z,y) + [z,y]\delta(x,y)=0.$$
Further, linearizing this relation on $y$ we get
$$[x,y]\delta(z,w) + [x,w]\delta(z,y) + [z,y]\delta(x,w) + [z,w]\delta(x,y) =0.$$
 Now, rewrite \eqref{q} so that  the roles of $y$ and $z$ are replaced:
$$[x,y]\cdot \delta(z,w) + [z,w]\cdot \delta(x,y) - [x,w]\cdot \delta(z,y) - [z,y]\cdot \delta(x,w)=0.$$
Summing up the last two relations 
we get
$$2\bigl([x,y]\cdot \delta(z,w) + [z,w]\cdot \delta(x,y)\bigr)=0,$$
which implies \begin{equation}\label{ena}[x,y]\cdot \delta(z,w) = [w,z]\cdot \delta(x,y).\end{equation}
In particular,
\begin{equation}\label{r}[[x,u],y]\cdot \delta(z,w) = [w,z]\cdot \delta([x,u],y).\end{equation}
 On the other hand, by the Jacobi identity,
$$[[x,u],y]\cdot \delta(z,w) = [[x,y],u]\cdot \delta(z,w) + [[y,u],x]\cdot\delta(z,w), $$
and hence, by \eqref{ena}, 
\begin{equation}\label{rr}[[x,u],y]\cdot \delta(z,w) = [w,z]\cdot \delta ([x,y],u) + [w,z]\cdot \delta([y,u],x).\end{equation}
Comparing \eqref{r} and \eqref{rr} we get 
\begin{equation}\label{dva} 
[w,z]\cdot \bigl( \delta ([x,y],u) + \delta([y,u],x) + \delta([u,x],y\bigr) =0. \end{equation}
 Now, applying that $\delta$ is a skew-symmetric biderivation it follows from \eqref{dva} by a direct calculation that
  $$2[w,z]\cdot \bigl(x\cdot \delta(y,u) + y\cdot\delta(u,x) + u\cdot\delta(x,y)\bigr)=0.$$
	Since 
	$$x\cdot \delta(y,u) + y\cdot\delta(u,x) = x\cdot \delta(y,u) - y\cdot\delta(x,u) = \delta([x,y],u) = -\delta(u,[x,y]),  $$
the desired conclusion follows.
\end{proof}

%Every skew-symmetric biderivation $\delta$ of a Lie algebra $L$ over a field of characteristic not $2$ satisfies 
%\begin{equation} \label{bd}
%[\delta(x,y),[z,w]] = [[x,y],\delta(z,w)]\quad\mbox{for all $x,y,z,w\in L$.}
%\end{equation}
%(This is known, one just has to find the most adequate reference.)

More can be said if $L$ is perfect.

\begin{lemma}\label{bl2}
Let $L$ be a perfect Lie algebra    and let $M$ be an $L$-module. If $\delta:L\times L\to M$ is a skew-symmetric biderivation, then
\begin{equation}\label{a} \delta(u,[x,y]) = u\cdot\delta(x,y), \forall u,x,y\in L.\end{equation}
\end{lemma}

\begin{proof}
For any $u,v,x,y\in L$, we have
\begin{align*}
&\delta([u,v],[x,y]) - [u,v]\cdot \delta(x,y) \\
= & u\cdot \delta(v,[x,y]) - v\cdot \delta(u,[x,y]) - u\cdot (v\cdot \delta(x,y)) + v\cdot (u\cdot \delta(x,y)) \\
=& u\cdot \Big(\delta(v,[x,y]) - v\cdot \delta(x,y)\Big) - v\cdot \Big(\delta(u,[x,y]) - u\cdot \delta(x,y)\Big).
\end{align*}
 Lemma \ref{bl}, together with assumption that $L=L'$, yields \eqref{a}.
\end{proof}

 We are now in a position to prove our fundamental theorem.

\begin{theorem}\label{mt}
Let $L$ be a  perfect Lie algebra   and let $M$ be an  $L$-module such that $Z_M(L)=\{0\}$.
Then every skew-symmetric biderivation $\delta:L\times L\to M$  is of the form
$\delta(x,y) = \gamma([x,y])$ where $\gamma\in {\rm Cent}(M)$.
\end{theorem}

\begin{proof}
 Define $\gamma:L=L' \to M$ by 
\begin{equation}\label{gama}\gamma( [x,y]) =  \delta(x,y), \forall x, y\in L.\end{equation}
Let us show that Lemma \ref{bl2} implies  that 
$\gamma$ is well-defined. Indeed,  assuming that $\sum_i [x_i,y_i]=0$, we have
$$0= \delta\Big(u,\sum_i[x_i,y_i]\Big)=\sum_i \delta(u,[x_i,y_i]) = u\cdot\Big(\sum_i \delta(x_i,y_i)\Big),$$
and hence  $\sum_i \delta(x_i,y_i)=0$ follows from
 $Z_M(L) =\{0\}$. 
We can now write \eqref{a} as
$
\delta(u,v) = u\cdot\gamma(v)$ for all $u, v\in L$.
Together with \eqref{gama} this shows that $\gamma\in {\rm Cent}(M)$.
%Thus $u\cdot\gamma(v)=\delta(u, v)=\gamma([u, v])$, showing that $\gamma\in {\rm Cent}(M)$.
\end{proof}

From now on we consider   skew-symmetric biderivations on a Lie algebra $L$, that is,  skew-symmetric biderivations from $L\times L$ to $L$.  We first record an immediate corollary to Theorem \ref{mt}.

\begin{corollary}\label{mco}
If $L$ is a  perfect and  centerless Lie algebra (in particular, if $L$ is simple), 
then every  skew-symmetric biderivation $\delta$ on $L$   is of the form
$\delta(x,y) = \gamma([x,y])$ where $\gamma\in {\rm Cent}(L)$.
\end{corollary}

Our goal now is to examine concrete situations to which our results are applicable.  To this end, we need  some auxiliary notions and results. 

Let $L$ be  an arbitrary Lie algebra.   Note that any skew-symmetric bilinear map $\delta:L\times L\to Z=Z_L(L)$ with $\delta(L, L')=0$  is a skew-symmetric biderivation. We will call it a {\it trivial biderivation} on $L$.  It is clear that every skew-symmetric biderivation having the range in $Z$ is automatically trivial.

If $\delta:L\times L\to L$  is  an arbitrary  skew-symmetric biderivation,
 for any $x, y\in L, z\in Z$ we have $$0=\delta([z,x],y)=[z, \delta(x,y)]+[\delta(z,y), x]=[\delta(z,y), x],$$
hence $\delta(Z, L)\subset Z$.
Consequently, setting $\bar L =L/Z$ we can define a 
skew-symmetric biderivation $\bar\delta:\bar L\times \bar L\to \bar L$ by
$$\bar \delta(\bar x, \bar y)=\overline{\delta( x, y)},$$
where $\bar x=x+Z\in \bar L$ for $x\in L$.

\begin{lemma}\label{mt1}
Let $L$ be a  Lie algebra.
Up to trivial biderivations  on $L$,  the map  $\delta\to \bar \delta$ is a 1-1 map from skew-symmetric biderivations on $L$ to skew-symmetric biderivations on $\bar L$. \end{lemma}

\begin{proof} Let $\delta_1, \delta_2$ be skew-symmetric biderivations on $L$  such that  $\bar \delta_1=\bar  \delta_2$. Then  $\delta=\delta_1-\delta_2$ is a skew-symmetric biderivation with $\delta(L, L)\subset Z$. Thus,  $\delta$ is a trivial biderivation on $L$.
\end{proof}

%Next we consider the situation where $L$ is centerless. 
 Any skew-symmetric biderivation $\delta: L\times L\to L$ satisfying $\delta(L',L')=0$ 
with range in $Z_L(L')$ will be  called a {\em special biderivation}.
%Note that any special biderivation $\delta$ satisfies $\delta(L, L')=0$.

\begin{example}
Every Lie algebra $L$ with nontrivial center and such that the codimension of $L'$ in $L$ is greater than $1$ has nonzero special biderivations. Indeed, 
by the codimension assumption, there exists a nonzero skew-symmetric bilinear functional $\omega:L\times L\to F$ such that $\omega(L,L')=\{0\}$, and taking any nonzero $z_0\in Z$, we have that $\delta(x,y) = \omega(x,y)z_0$ is a nonzero special biderivation which is also trivial.  
\end{example}

The next example shows that there exist special biderivations $\delta$  that are  not of the form $\delta(x, y)=\gamma([x, y])+\delta_0(x, y)$ where  $\gamma\in {\rm Cent}(L)$ and $\delta_0$ is a trivial biderivation.

\begin{example} \label{special}
Let $\mathcal F$ be the free Lie algebra in  variables $x_1,x_2,x_3$. Denote by $\mathcal I$ the ideal of $\mathcal F$ generated by all elements of the form 
$[[x_1,f_1],f_2]+[[x_1,f_2],f_1]$ where $f_1, f_2\in \mathcal F$, and by $\mathcal J$ the ideal of $\mathcal F$ generated by all elements of the form 
$[[x_1,g_1],g_2]$ where $g_1, g_2\in \mathcal F'$.

Set $L=\mathcal F/(\mathcal I+\mathcal J)$. It is clear that $a=\bar x_1=x_1+\mathcal I+\mathcal J$ is a nonzero element in $L$. Define the bilinear map $\delta(x, y)=[[a, x], y]$ for all $x, y\in L$. It is easy to see that $\delta$ is a special biderivation on $L$.

Assume there are $\gamma\in {\rm Cent}(L)$ and a trivial biderivation $\delta_0$ on $L$ such that 
$$\delta(x, y)=\gamma([x, y])+\delta_0(x, y).$$ For $x, y, u\in L$, we then have
$$\aligned &[[[a, [x, y]], u]=\delta([x, y], u)\\
=&\gamma([[x, y], u])+\delta_0([x, y], u)
=[\gamma([x, y]), u]+\delta_0([x, y], u)\\
=&[\delta(x, y), u]+\delta_0([x, y], u)
=[[[a, x], y], u]+\delta_0([x, y], u),
\endaligned
$$
yielding that $[[[a, y], x], u]\in Z$, the center of $L$. We claim that this is not the case. Specifically, we will show that
 \begin{equation}\label{jk}[[[[\bar x_1, \bar x_2], \bar x_3], \bar x_3], \bar x_3]\ne0 ,\end{equation} so that $[[\bar x_1, \bar x_2], \bar x_3], \bar x_3]\notin Z$. 

As usual, we can define monomials, the degrees $\deg_1, \deg_2, \deg_3$ of a monomial with respect to $x_1, x_2, x_3$ respectively, and homogeneous elements in $\mathcal F$.
Then every element in $\mathcal I$ (or $\mathcal J$) can be written as a sum of homogeneous elements in $\mathcal I$  (or $\mathcal J$). It is easy to see that there are no nonzero elements in $\mathcal J$ with $\deg_1=1, \deg_2=1$ and $ \deg_3=3$.
The only homogeneous elements  with $\deg_1=1, \deg_2=1$ and $ \deg_3=3$ 
in $\mathcal I$ are 
\begin{align*}
&[[[[ x_1,   x_2],   x_3]+ [[ x_1,   x_3],   x_2],x_3], x_3]\\
&=[[[[ x_1,   x_2],   x_3],x_3], x_3]+[[[[x_1, x_3], x_2], x_3], x_3],\\
 &[[[ x_1,   x_3],   x_3], [x_2, x_3]]=[[[ x_1,   x_3],   x_3], x_2], x_3]-[[[ x_1,   x_3],   x_3], x_3], x_2],\\
 &[[[[ x_1,   x_3],   x_3], x_2], x_3].
\end{align*}
Since $$[[[[ x_1,   x_2],   x_3],x_3], x_3],\,\,\, [[[[x_1, x_3], x_2], x_3], x_3],$$ 
$$ [[[[ x_1,   x_3],   x_3], x_2], x_3],\,\,\,[[[[ x_1,   x_3],   x_3], x_3], x_2],$$   are linearly independent in $\mathcal F,$ it follows that $[[[[ x_1,   x_2],   x_3],x_3], x_3]$ does not lie in $\mathcal I$, and hence neither in $\mathcal I+\mathcal J$. This proves \eqref{jk}.
Consequently, $\delta$ is not of the form $\delta(x, y)=\gamma([x, y])+\delta_0(x, y)$. 
\qed
\end{example}

%It should be pointed out that special biderivations do not necessarily arise from the centroid. As a trivial example, if $L$ is an %Abelian Lie algebra, then every
%skew-symmetric linear map from $L\times L$ to $L$ is a special biderivation. However, if it is nonzero, then it is not of the %form $\gamma([x,y])$ with $\gamma\in{\rm Cent}(L)$.

Every skew-symmetric biderivation $\delta:L\times L\to L$ satisfies 
\begin{equation}\label{a'''}
\delta(u,[x,y]) =[x,\delta(u,y)]+[ \delta(u,x), y]\in L', \forall x, y, u\in L.
\end{equation} 
  Thus we have a 
skew-symmetric biderivation $\delta': L'\times  L'\to   L'$ by restricting $\delta$ to $L'\times L'$.

\begin{lemma}\label{perfect}   Let $L$ be a  centerless Lie algebra. 
\begin{enumerate}
\item[{\rm (a)}] Up to a special biderivation, any skew-symmetric biderivation $\delta$ on $L$ is an extension of a   unique skew-symmetric  biderivation on $ L'$. 

\item[{\rm (b)}] If $L'$ is further perfect, then $L$ has no nonzero special biderivation.
\end{enumerate}
\end{lemma}

\begin{proof}  (a)
 Let $\delta_1, \delta_2$ be biderivations on $L$  such that  $ \delta'_1= \delta'_2$. Let  $\delta=\delta_1-\delta_2$. Then $\delta(L', L')=0$.  Taking  $u, y\in L'$ in \eqref{a'''}  we see that 
$$[ \delta(u,x), y]=0, \forall x\in L, y, u\in L',$$
i.e.,
$\delta(L,L')\subset Z_L(L')$.  Applying this to Lemma \ref{bl}, we see that 
$[L, \delta(L, L)] \subset  Z_L(L')$. 

Using (\ref{ena}) we have 
$$0=[[L,L],\delta(L, L')]=[\delta(L,L), [L, L']].$$
Then for any $x, y, z, u,v\in L$ we have 
$$0=[[[x, y], z], \delta(u, v)]=[[[x, y], \delta(u, v)], z]+[[x, y], [z ,\delta(u, v)]]=[[[x, y], \delta(u, v)], z].$$
Since $L$ is centerless, we deduce that 
$\delta(L, L)\subset Z_L(L')$.
Thus,  $\delta$ is a special biderivation of $L$.

(b) Let $\delta$ be a special biderivation of $L$. From the beginning of the proof of (a) we see that $\delta(L,L')\subset Z_L(L')$. Since $L$ is perfect and centerless, it follows that $\delta(L,L')=0$.

%It is enough to show further that the above $\delta$ is $0$ if $L'$ is perfect. We continue from the proof of (a). Since 
%$L'$ is perfect, from $\delta(L,L')\subset Z_L(L')$ we see that $\delta(L,L')=0$.

Assuming that  $\delta\ne0$, there are $x_1, x_2\in L$ such that $\delta(x_1, x_2)=z_{12}\ne0$. Since $L$ is centerless we can find $x_3\in L$ so that $[x_3, z_{12}]=z\ne0$. Let $\delta(x_1, x_3)=z_{13}, \delta(x_2, x_3)=z_{23}$.
From 
$$\aligned &0=\delta([x_1, x_3], x_2)=[x_1,\delta( x_3, x_2)]-z=-[x_1, z_{23}]-z,\\
&0=\delta([x_1, x_2], x_3)=[z_{13}, x_2]+[x_1, z_{23}],\\
&0=\delta([x_2, x_3], x_1)=-[x_2,z_{13}]+z,
\endaligned 
$$
we deduce that $-z=[x_1, z_{23}]=-[z_{13}, x_2]=z$, a contradiction.
Thus, $\delta=0$.
\end{proof}

Let us explain an algorithm for finding
 all  skew-symmetric biderivations on a Lie algebra  $L$
 using  Lemmas \ref{mt1}  and \ref{perfect}. We have a sequence of quotient Lie algebras:
\begin{equation}\label{quo1} L_{(1)}=L, L_{(2)}=L_{(1)}/Z(L_{(1)}), \cdots,  L_{(r+1)}=L_{(r)}/Z(L_{(r)}).\end{equation}
If there is  $r\in\mathbb{N}$ such that $Z(L_{(r)})=0$, then repeatedly applying Lemma \ref{mt1} to the above sequence backward, we reduce the problem of finding   skew-symmetric  biderivations on $L$ to the problem of finding   skew-symmetric biderivations on the centerless Lie algebra $L_{(r+1)}$. If $L_{(r+1)}$ is also perfect, then  using Corollary \ref{mco} we have all  biderivations on 
$L_{(r+1)}$. We are done in this case.
If $L_{(r+1)}$ is not perfect, using Lemma \ref{perfect} we  reduce the problem of finding   skew-symmetric biderivations on $L_{(r+1)}$ to the problem of finding  skew-symmetric  biderivations on the   Lie algebra $L'_{(r+1)}$. Now we repeat the procedure  based on (\ref{quo1}) with $L$ replaced by $L'_{(r+1)}$, and continue this algorithm.
 
%This algorithm works for various Lie algebras, as we will see.
We will now apply our results to concrete examples.
% we first recall the following well-known result from \cite{D}.

%\begin{theorem}\label{Shur} Let $F$ be an algebraically closed field, 
% $L$ be a Lie algebra over $F$, and $M$ be an $L$-module such that ${\rm{Card}}(F)>\dim (M)$. Then any $L$-module endomorphism of  $M$ is a scalar.
%\end{theorem}

\begin{example}   
Let $L$ be a  simple Lie algebra over  an algebraically closed field $F$ such that ${\rm{Card}}(F)>\dim (L)$. 
As is well-known, every  $L$-module endomorphism of  $L$ is a scalar map \cite{D}.
% By Theorem \ref{Shur} we know that centroids of $L$ are scalar maps from $L$ to $L$.
Therefore, by Corollary \ref{mco}, every 
  skew-symmetric   biderivation $\delta$ of $L$ 
	is of the form  $\delta(x,y) = \lambda [x,y]$, $x,y\in L$, for some $\lambda\in F$. 
	
	This example, in particular, covers the results in 
\cite{Chen2016}. \qed
\end{example}

\begin{example} \label{ex2}   For $a, b\in \mathbb{C}$ with $(a, b)\notin \mathbb{Z}\times\{0, -1\}$,
the Lie algebra $${W}(a,b)={\rm span}_{\mathbb{C}}\{L_m, I_m\,|\,m\in \mathbb{Z}\}$$ is an infinite-dimensional Lie algebra over $\mathbb{C}$ equipped with the following brackets:
\begin{align*}
[L_m,L_n ]&=(n-m)L_{m+n}, \\
 [L_m,I_n]&=(n+a+bm)I_{m+n},  \\ [I_m,I_n]&=0,
\end{align*}
for all $m,n\in \mathbb{Z}$.
These Lie algebras $W(a, b)$ are perfect and centerless. Let  $L=\oplus _{i\in\mathbb{Z}}\mathbb{C}L_i$ and $I=\oplus _{i\in\mathbb{Z}}\mathbb{C}I_i$. Then $I$ is an ideal of $W(a, b)$ and $L$ is a subalgebra which is isomorphic to the Witt algebra. As indecomposable  $L$-modules, $L$ and $I$ are not isomorphic, and $L$-module homomorphisms of $L$ (and $I$) are scalar maps.

Let $\gamma \in {\rm Cent}(W(a, b))$. Then $$\gamma (I)\subset \gamma ([L, I])=[\gamma (L), I]\subset I,$$
 and $$\gamma ([L_m, I_n])=[\gamma (L_m), I_n]=[L_m, \gamma (I_n)].$$ 
 Since $\gamma$ can be considered as an $L$-module homomorphism, we see that $\gamma(L)\subset L$.
 There exist $c_1,c_2\in \mathbb{C}$ such that $\gamma (L_m)=c_1L_m, \gamma (I_m)=c_2I_m$. Furthermore, $c_2=c_1$.
So the centroid of $W(a, b)$ consists of scalar maps from $W(a, b)$ to $W(a, b)$.
 Corollary \ref{mco} therefore tells us that for every
  skew-symmetric   biderivation $\delta$ of $W(a,b)$  there exists $\lambda\in\mathbb{C}$ such that
  $\delta(x,y) = \lambda [x,y]$, $x,y\in W(a,b)$.
 %From Theorem \ref{mt}, then for any   skew-symmetric   biderivation $\delta$ of $W(a, b)$, there is $\lambda\in \mathbb{C}$ depending on $\delta$ so that  $\delta(x,y) = \lambda [x,y]$ for all $x,y\in W(a, b)$.  
\qed
\end{example}

\begin{example} \label{ex2'}   Consider
the Lie algebra $W={W}(0,0)$ which is defined using the same brackets as in Example \ref{ex2} with $a=b=0$.
This Lie algebra has center $Z=\mathbb{C}I_0$, and $\bar W=W/Z$ is   perfect and centerless. 
By similar arguments we see that any map in $ {\rm Cent}(\bar W)$ is a scalar map.
 Corollary \ref{mco} therefore tells us that for every
  skew-symmetric   biderivation $\delta$ of $\bar W$  there exists $\lambda\in\mathbb{C}$ such that
  $\delta(x,y) = \lambda [x,y]$, $x,y\in\bar W$. Since $W$ is perfect, it follows from Lemma \ref{mt1}  that  
  every
  skew-symmetric   biderivation $\delta$ of $W$ is of the form   $\delta(x,y) = \lambda [x,y]$, $x,y\in W$, for some 
$\lambda\in\mathbb{C}$.
 %From Theorem \ref{mt}, then for any   skew-symmetric   biderivation $\delta$ of $W(a, b)$, there is $\lambda\in \mathbb{C}$ depending on $\delta$ so that  $\delta(x,y) = \lambda [x,y]$ for all $x,y\in W(a, b)$.  
\qed
\end{example}

In the next example we will derive the usual conclusion that all skew-symmetric biderivations of the Lie algebra $L$ in question are of the form $\delta(x,y)=\lambda[x,y]$ with $\lambda\in \mathbb C$. It is interesting, however, that this will be derived from the description of skew-symmetric biderivations of $\bar L =L/Z$ which is more involved (that is,
the centroid of $\bar L$  contains more than just scalar maps).

\begin{example}\label{ex2''}  Consider the  complex Lie algebra  $$\tilde W(0,-1)={\rm span}_{\mathbb{C}}\{L_m, I_m, c_1, c_2, c_3\,|\,m\in \mathbb{Z}\}$$ 
equipped with the following brackets:
\begin{align*} [L_m, L_n]=&(n-m)L_{m+n}+\delta_{m, -n}\frac{m^3-m}{12}c_1, \\
[L_m,I_n]=&(n-m)I_{m+n}+\delta_{m, -n}\frac{m^3-m}{12}c_2,  \\  
[I_m,I_n]=&\delta_{m, -n}\frac{m^3-m}{12}c_3,\end{align*}
and $c_i$ being central elements. In fact, the center $Z$ of $\tilde W(0,-1)$ is equal to 
 $\rm{span}_{\mathbb{C}}\{c_1, c_2, c_3\}$. The Lie algebra $\tilde W(0,-1)/Z= W(0,-1)$ is defined using the same brackets as in Example \ref{ex2} with $(a, b)=(0,-1)$.
The Lie algebra $W(0, -1)$ is perfect and centerless. 
Let $L=\oplus _{i\in\mathbb{Z}}\mathbb{C}L_i\subset W(0,-1)$ and $I=\oplus _{i\in\mathbb{Z}}\mathbb{C}I_i\subset W(0,-1)$. Then $I$ is an ideal of $W(0, -1)$ and $L$ is a subalgebra.  As $L$-modules, $L$ and $I$ are isomorphic.

Let $\gamma \in {\rm Cent}(W(0,-1))$. Then $\gamma (I)\subset \gamma ([L, I])=[\gamma (L), I]\subset I$. 
There are $a, b, c\in \mathbb{C}$ such that $\gamma (L_m)=aL_m+bI_m, \gamma (I_m)=cI_m$. Since  $$\gamma ([L_m, I_n])=[\gamma (L_m), I_n]=[L_m,\gamma (I_n)],$$  we deduce that $c=a$.
We denote this element in ${\rm Cent}(W(0, -1))$  by $\gamma _{a,b}$. So ${\rm Cent}(W(0,-1))=\{\gamma _{a,b}\,|\,a,b\in\mathbb{C}\}$. From  Corollary \ref{mco} we see that for any   skew-symmetric   biderivation $\delta$ of $W(0, -1)$, there are $a, b\in \mathbb{C}$ such that  $$\delta(x,y) = \gamma _{a,b}( [x,y]), \forall x,y\in W(0, -1).$$ 

It is easy to see that there is  $g_a\in {\rm Cent}(\tilde W(0,-1))$ such that $\bar g_a=\gamma _{a,0}$ (actually $g_a$ is the scalar map determined  by $a$).

Next suppose there exists a  skew-symmetric biderivation $h$ on $\tilde W(0,-1)$ such that $$\bar h(\bar x, \bar y)=\gamma _{0, b}([\bar x, \bar y]).$$ Then $h(L, I)\subset Z$, and further
$$h(L, I)=h(L', I)=[h(L, I), L]+[L, h(L, I)]=0.$$ We may assume that 
$$h(L_m, L_n)=(n-m)bI_{m+n}+C_{m,n}, {\rm{\text{ where }}} C_{m,n}\in Z.$$
From $$\aligned &(n-m)\left((r-m-n)bI_{m +n+r}+C_{m+n, r}\right)=h([L_m ,L_n], L_r) \\
=&[h(L_m, L_r), L_n]+[L_m, h(L_n, L_r)]\\
=&(r-m)b((n-r-m)I_{m+n+r}+\delta_{m+n, -r}\frac{n-n^3}{12}c_2)\\
&+(r-n)b((n+r-m)I_{m+n+r}+\delta_{m+n, -r}\frac{m^3-m}{12}c_2),
\endaligned$$
we obtain that 
$$(n-m) C_{m+n, r} 
=(r-m)b\delta_{m+n, -r}\frac{n-n^3}{12}c_2
+(r-n)b\delta_{m+n, -r}\frac{m^3-m}{12}c_2.$$
Letting $n=-m-r$ we have
$$-(r+2m) C_{-r, r} 
=\big((r-m)((m+r)^3-m-r)
-(2r+m)(m^3-m))\big)bc_2/12.$$
Note that $C_{-r, r} =0$. Since $m$ is arbitrary, it follows that 
$b=0$. Thus, if $b\ne 0$, there is no  biderivation $h$ on $\tilde W(0,-1)$ such that $\bar h(\bar x, \bar y)=\gamma _{0, b}([\bar x, \bar y])$.

By Lemma \ref{mt1}, every   skew-symmetric   biderivation $\delta$ of $\tilde W(0,-1)$
is of the form  $\delta(x,y) = \lambda [x,y]$, $x,y\in \tilde W(0,-1)$, for some
 $\lambda\in \mathbb{C}$. 

This and the previous two examples recover all  results in  \cite{Hanw}. \qed
\end{example}

\begin{example}\label{ex3}  The Schr\"odinger-Virasoro Lie algebra $S$ is the infinite-dimensional Lie
algebra with $\mathbb{C}$-basis $\{L_m,\,Y_p\,,\,M_n\,|\,m,\,n, p\in
\mathbb{Z} \}$ and Lie brackets,
$$\aligned
&[L_m,\,L_{n}]=(n-m)L_{n+m} ,
\\
&[L_m,\,Y_p]=(p-\frac{m}{2})Y_{p+m}, \\[2pt]
& [L_m,\,M_n]=nM_{n+m}, \\[3pt]
&[Y_p,\,Y_{n}]=(n-p)M_{n+p}\,, \\[3pt]
&[Y_p,\,M_n]=[M_n,\,M_{m}]=0. \endaligned$$
We know that $S$ is  perfect with center $Z=\mathbb{C} M_0$, and that  $\bar S=S/Z$ is perfect and centerless. Let  $L=\oplus _{i\in\mathbb{Z}}\mathbb{C}L_i$, $Y=\oplus _{i\in\mathbb{Z}}\mathbb{C}Y_i$  and $M=\oplus _{i\in\mathbb{Z}\setminus\{0\}}\mathbb{C}M_i$. 
 
 Let $\gamma \in {\rm Cent}(\bar S)$. Since  $L$, $M$, $Y$  are not isomorphic to each other as indecomposable $L$-modules, 
there are $a, b, c\in \mathbb{C}$ such that $$\gamma (L_m)=aL_m,\,\,\, \gamma (M_m)=bM_m,\,\,\, \gamma (Y_m)=bY_m.$$ From $$\gamma ([L_m, M_n])=[\gamma (L_m), M_n]=[L_m, \gamma (M_n)],$$ 
$$\gamma ([L_m, Y_n])=[\gamma (L_m), Y_n]=[L_m, \gamma (Y_n)],$$   we deduce that $a=b=c$. Thus $\gamma$ is a scalar map.  
By Corollary \ref{mco}, for any   skew-symmetric   biderivation $\delta$ of $\bar S$, there is $\lambda\in \mathbb{C}$ such that  $\delta(x,y) = \lambda [x,y]$ for all $x,y\in \bar S$.  Since $S$ is perfect, Lemma \ref{mt1} implies that every   skew-symmetric   biderivation $\delta$ 
is of the form $\delta(x,y) = \lambda [x,y]$, $x,y\in S$, for some 
 $\lambda\in \mathbb{C}$.

This example recovers all results in 
\cite{WD1}. \qed
\end{example}

\begin{example}\label{ex4}   Let $q\in\mathbb{C}$. 
The Block Lie algebra $B(q)$ is the Lie algebra with a basis $\{L_{m,i}\,|\, m, i\in \mathbb{Z}\}$  subject to the following Lie brackets
$$
[L_{m,i},L_{n,j}]=(n(i+q)-m(j+q))L_{m+n,i+j}, \ \forall\ i,j,m,n\in\mathbb{Z}.$$
Some of these Lie algebras are not perfect or centerless. Anyway $B(q)'/Z$ is a simple Lie algebra. So for any  skew-symmetric   biderivation $\delta$ of $B(q)'/Z$, there is $\lambda\in \mathbb{C}$ such that  $\delta(x,y) = \lambda [x,y]$ for all $x,y\in B(q)'/Z$.  

Now we use   Lemma \ref{mt1}. If $q\ne0$, for any  skew-symmetric   biderivation $\delta$ of $B(q)'$, there is $\lambda\in \mathbb{C}$ such that  $\delta(x,y) = \lambda [x,y]$ for all $x,y\in B(q)'$. 
Using  Lemma \ref{perfect}, if $q\ne0$, there is no special biderivations, hence,  for any  skew-symmetric   biderivation $\delta$ of $B(q)$ there is $\lambda\in \mathbb{C}$ such that  $\delta(x,y) = \lambda [x,y]$ for all $x,y\in B(q)$. 

If $q=0$, we know that $B(0)=B(0)'\oplus\mathbb{C}L_{0,0}$ as ideals. It is easy to see that for any  skew-symmetric   biderivation $\delta$ of $B(0)$ there is $\lambda\in \mathbb{C}$ such that  $\delta(x,y) = \lambda [x,y]$ for all $x,y\in B(0)$.

This example recovers the results  on  skew-symmetric   biderivations  in 
\cite{GLZ}. 
 \qed
\end{example}

We have thus seen that our methods cover a variety of results from the literature. On the other hand, obviously they do not work  for  Lie algebras $L$ such that the sequence (\ref{quo1}) does not terminate.

% \section{Symmetric   biderivations}
 
We close the section by discussing general, not necessarily skew-symmetric, biderivations. By a {\em biderivation} we mean, of course,  a bilinear map $\delta:L\times L\to M$
such that the maps $x\mapsto \delta(x,z)$ and $x\mapsto \delta(z,x)$ are derivations for every $z\in L$.
Biderivations that are skew-symmetric seem to be more natural than others, in particular because of their connection to the centroid and commuting linear maps. 
%The fact that commuting linear maps give rise to  biderivations that are skew-symmetric indeed gives a good excuse for studying only these ones;
 However,  
 general biderivations are also
 mathematically challenging and  therefore   deserve some attention.

Every biderivation $\delta$ can be written as the sum of a symmetric biderivation and a skew-symmetric biderivation, namely,
$$\delta(x,y) =\frac{\delta(x,y) + \delta(y,x)}{2} +  \frac{\delta(x,y) - \delta(y,x)}{2}.$$
We can therefore focus on symmetric biderivations. Finding examples on  Abelian Lie algebras $L$ is trivial: every symmetric bilinear map from $L\times L$ to $L$ is a symmetric biderivation. 
  %A simple example can be obtained as follows. Take a derivation $d$ on a commutative associative  algebra $A$. Then $\delta(x,y)=d(x)d(y)$ is a symmetric biderivation from $A\times A$ to $A$. Thus, it is easy to find nonzero symmetric biderivations on Abelian Lie algebras.
 On the other hand, in a Lie (or associative) algebra with trivial (or small) center, one normally
expects that a symmetric biderivation is always $0$. In view of Theorem \ref{mt}, it is tempting to conjecture that there are no nonzero symmetric biderivations  from $L\times L$ to $M$ if $L$ is perfect and $Z_M(L)=\{0\}$. However, the following example (which admittedly came as a surprise to us)  shows that this is not the case -- not even when $L=\mathfrak{sl}_2$!

\begin{example}  For convenience we take a basis $d_1, d_0, d_{-1}$ of the Lie algebra  $\mathfrak{sl}_2$ subject to the brackets
 $$[d_i,d_j]=(j-i)d_{i+j}, \forall i,j\in\{0, 1, -1\},$$
 where $d_k=0$ if $k\notin\{0,1,-1\}$.
  For $a, b\in\mathbb{C}$, let $M(a,b)=\oplus_{i\in\mathbb{Z}}\mathbb{C}v_i$ be the $\mathfrak{sl}_2$-module given by
 $$d_i\cdot v_j=(j+a+bi)v_{i+j}.$$
Note that $M(a,b)$ is a simple $\mathfrak{sl}_2$-module if  and only if $a\notin  \mathbb{Z}$ or $b\notin \mathbb{Z}$. Also,  $Z_{M(a,0)}(\mathfrak{sl}_2)={\rm{span}}\{v_{-a}\}$ if $a\in\mathbb{Z}$, and otherwise $Z_{M(a,0)}(\mathfrak{sl}_2)=\{0\}$.

Take an arbitrary $k\in \mathbb{Z}$. It is easy to see that  the  bilinear map $\delta_k: \mathfrak{sl}_2 \times \mathfrak{sl}_2\rightarrow M(a,0)$  determined by
\begin{eqnarray*}
 \delta_k(d_m,d_n)= v_{m+n+k}, \forall i,j\in\{0, 1, -1\},
\end{eqnarray*}
is a symmetric biderivation.
%We see that $\delta_k(\mathfrak{sl}_2, \mathfrak{sl}_2)=\{v_{k}, v_{k\pm1}, v_{k\pm2}\}$  which is not a submodule of $M(a, 0)$.
Similarly,
for any $k\in \mathbb{Z}$,   $\delta'_k: \mathfrak{sl}_2 \times \mathfrak{sl}_2\rightarrow M(a,1)$  determined by 
\begin{eqnarray*}
 \delta'_k(d_m,d_n)= (m+n+k+a)v_{m+n+k}, \forall i,j\in\{0, 1, -1\}
\end{eqnarray*}
is a symmetric biderivation. \qed
%We see that  
%$\delta_k(\mathfrak{sl}_2, \mathfrak{sl}_2)=\{v_{k}, v_{k\pm1}, v_{k\pm2}\}$ (except for $a\in\mathbb{Z}$) is not submodule of $M(a, 1)$.
\end{example}

We leave as an open question whether or not there exists a simple Lie algebra $L$ (over a field of characteristic not $2$) that admits a nonzero symmetric biderivation from $L\times L$ to $L$.

\section{Commuting linear maps}

Let $M$ be a module over a Lie algebra $L$.  It is clear that every $\gamma\in{\rm Cent}(M)$ satisfies $x\cdot \gamma(x)=0$ for each $x\in L$. 
 We will show that under a mild  assumption, this condition is characteristic for the centroid., i.e., commuting linear maps $f$ from $L$ to $M$ belong to $\rm{Cent}(M)$.

\begin{lemma}\label{ll}
Let $L$ be a Lie algebra  and let $M$ be an  $L$-module.
If $f:L\to M$ is a commuting linear map, then 
$$[w,z]\cdot\Big(u\cdot \big( f([x,y]) - x\cdot f(y)\big)\Big)=0$$
for all $x,y,u,w,z\in L$.
\end{lemma}

\begin{proof}
Linearizing $x\cdot f(x)=0$ we  get 
$x\cdot f(y)=- y\cdot f(x)$. This shows that 
the map $\delta:L\times L\to M$, $\delta(x,y)= x\cdot f(y)$ is a skew-symmetric  biderivation.
According to Lemma \ref{bl}, $$[w,z]\cdot\Big(\delta(u,[x,y]) - u\cdot\delta(x,y)\Big)=0$$ for all $x,y,u,w,z\in L$. Since 
$\delta(x,y)=x\cdot f(y)$ and $\delta(u,[x,y]) =  u\cdot f([x,y])$, the result follows.
\end{proof}

The following theorem follows immediately from
Lemma \ref{ll}.

\begin{theorem}\label{c1}
Let $L$ be a Lie algebra  and let $M$ be an  $L$-module such that $Z_M(L')=\{0\}$.
If  $f:L\to M$ is a commuting linear map, then $f\in {\rm Cent}(M)$.
\end{theorem}

In the special case where $M=L$, Theorem \ref{c1} gets the following form.

\begin{corollary}\label{c2}
Let $L$ be a Lie algebra such  that $Z_L(L')=\{0\}$. Then every commuting linear map
$f:L\to L$ belongs to ${\rm Cent}(L)$.
\end{corollary}

We remark that the set of commuting linear maps of Lie algebras $L$ was  studied in \cite{LL} under the name  of {\em quasi-centroid}. It was shown, in particular,  that the quasi-centroid of $L$ coincides with the centroid   in case $L$ is finite dimensional, centerless, and perfect \cite[Theorem 5.28]{LL}.
Corollary \ref{c2} is obviously considerably stronger. In particular, it shows that the assumption that $L$ is finite dimensional is superfluous.

If $L$ has a nontrivial center, then every map with the range in $Z$ is commuting. Moreover, it does not lie in Cent$(L)$ in case it does not vanish on $L'$.  This justifies  that  Corollary \ref{c2} deals with  centerless Lie algebras. However, we have assumed more than that.  To justify the assumption that $L'$ has trivial centralizer in $L$, consider the following example. It is a modification of an example from \cite{Cheung}. 

\begin{example}\label{3.4}
Let $L$ be the Lie algebra consisting of all $4\times 4$ matrices of the form
 $$\left[\begin{matrix} x_{11} & x_{12} &  x_{13} & x_{14}\cr  
0& x_{11} &  0 & x_{24}\cr
0 & 0 &  0 & x_{34}\cr
0 & 0 &  0 & 0\cr
\end{matrix} \right]$$
with $x_{ij}\in F$. One easily checks that $L$ is centerless (but $L'$ is not), and that  $f:L\to L$ given by
 $$\left[\begin{matrix} x_{11} & x_{12} &  x_{13} & x_{14}\cr  
0& x_{11} &  0 & x_{24}\cr
0 & 0 &  0 & x_{34}\cr
0 & 0 &  0 & 0\cr
\end{matrix} \right]\mapsto 
\left[\begin{matrix} 0 & x_{13} &  0 & 0\cr  
0& 0&  0 & 0\cr
0 & 0 &  0 & x_{24}\cr
0 & 0 &  0 & 0\cr
\end{matrix} \right]
$$
is a linear commuting map. However, $f\notin {\rm Cent}(L)$ since, for example, $$0=f([e_{13},e_{24}]) \neq  [e_{13}, f(e_{24})] = e_{14}.$$
Here, $e_{ij}$ are matrix units.\qed
\end{example}

As we mentioned above, 
%If $L$ has a nonzero center $Z$, then we have another type of commuting linear maps. Namely, 
every {\em central} linear map, i.e., a map with the range in the center $Z$ of $L$, is trivially commuting.  Note also that the sum of commuting maps is again commuting. Thus, if $\gamma\in{\rm Cent}(L)$ and $\mu$ is a central map, then $f=\gamma + \mu$ is commuting.

 The next simple lemma connects the problem of describing skew-symmetric biderivations with the problem of describing commuting linear maps.

\begin{lemma}\label{clm} Let $L$ be a Lie algebra. If every skew-symmetric biderivation $\delta$ on L is of the form $\delta(x,y)= \gamma([x,y])$, then every commuting linear map f on L is of the form $ f = \gamma + \mu$,
where $\gamma\in {\rm Cent}(L)$  and $\mu$ is a central linear map. \end{lemma}

\begin{proof}
If $f:L\to L$ is a commuting linear map, then 
$[f(x),y]= -[f(y),x]$ for all $x,y\in L$, and hence  $\delta(x,y)=[f(x),y]$ is a skew-symmetric biderivation. This readily implies the conclusion of the lemma.
\end{proof}

Using this lemma together with the description of biderivations on various Lie algebras obtained in  the preceding section, we can now also describe commuting linear maps on all these Lie algebras.  

\begin{example}  % If we apply this lemma to all the Lie algebras in Examples \ref{ex2} to \ref{ex4} in Sect. 2,   we obtain that e
Every commuting linear map $f$ on $W(a, b)$ for $a, b\in \mathbb{C}$, $\tilde {W}(0, -1)$, $S$, $\bar S$ and $B(q)$ for $q\in\mathbb{C}$,  is of the form $ f = \gamma + \mu$,
where $\gamma$ lies in the centroid and $\mu$ is a central map. \qed
\end{example}

We now give an example of 
a commuting linear map that is not a sum of a  map in ${\rm Cent}(L)$ and a central map.
It is particularly interesting that this map is a derivation.

\begin{example} \label{irre} Let $L$ be the Lie algebra from Example \ref{special}. We know that  the derivation  $f(x)=[a,x]$ 
 is a commuting linear map. Suppose $f$ is a sum of a   map in ${\rm Cent}(L)$ and a central map. Then  $$[a,[x,y]]=[x,[a,y]] \mod Z, \forall x,y\in L,$$ and hence $[[a,x],y]\in Z$ for all $x,y\in L$. However, 
%if there are $x,y\in L$ so that $[[a,x],y]\notin Z$, then the commuting linear map $f$ is not a sum of a  central map and map in ${\rm Cent}(L)$.
from Example \ref{special} we know that this is not true; specifically, this follows from $[[\bar x_1, \bar x_2], \bar x_3]\notin Z$.
Thus, $f$ is not a sum of  a  map in ${\rm Cent}(L)$ and a central map.
\qed
\end{example}

Finally, we will propose an algorithm for computing  commuting linear map of a Lie algebra without using biderivations. 
Although we are  interested in commuting linear maps from $L$ to $L$, this algorithm involves commuting linear maps from $L$ to an $L$-module $M$ (and hence gives one of justifications for working in a more general framework involving modules). 
In view of  Theorem \ref{c1} and Corollary \ref{c2}, we are  now interested in the case where $Z_M(L')\ne0$. 

 Any linear map $f:L\to Z_M(L)$ is   a commuting map. We  call it a
{\em central map}.
 A commuting  linear map $f:L\to M$ will be  called  a 
{\em special commuting linear map} if  
$f(L')=0$ and  $f(L)\subset Z_M(L')$. %If we write $Z_M(L')=Z_1\oplus Z_M(L)$, we may further assume that $f(L)\subset Z_1$ after subtracting  a central map from $f$.%Nonzero special commuting linear maps are not central.

For any commuting linear map $f: L\to M$ we can define another commuting linear map $\tilde  f: L\to M/Z_M(L')$ by $$\tilde f(x)= f(x)+Z_M(L').$$
The following lemma then holds.

\begin{corollary}\label{mt1'}
Let $L$ be a   Lie algebra and $M$ be an $L$-module.
Then  $f\to  \tilde f$ is,  up to sums of special commuting and central linear maps, a 1-1 map from 
commuting linear maps from $L$ to $M$ to commuting linear maps from $L$ to $M/Z_M(L')$.
\end{corollary}

\begin{proof} Let $f_1, f_2$ be commuting linear maps on $L$  such that  $\tilde f_1=\tilde  f_2$. Let  $f=f_1-f_2$. Then $f( L)\subset Z_M(L')$. From $0=L'\cdot f(L)=-L\cdot f(L')$, we see that $f(L')\subset Z_M(L)$. By subtracting a central linear map from $f$ we may assume that $f(L')=0$.   Now this $f: L\to M$ is a special commuting linear map.\end{proof}

For an $L$-module $M$ we have a sequence of quotient modules:
\begin{equation}\label{quo} M=M_1, M_2=M_1/Z_{M_1}(L'), \cdots,  M_{r}=M_{r-1}/Z_{M_{r-1}}(L').\end{equation}
If there is  $r\in\mathbb{N}$ such that $Z_{M_r}(L')=0$, we can start to find commuting maps $f_r: L\to M_r$. Then repeatedly using Corollary \ref{mt1'} we obtain all commuting maps  $f: L\to M$. Let us point out that   $Z_{M_r}(L')=0$ implies, by Theorem \ref{c1}, that every commuting map $f_r: L\to M_r$ lies in Cent$(M_r)$.

We conclude the paper by an example illustrating this algorithm.

\begin{example} \label{g} Let $\mathfrak{g}$ be a finite dimensional simple Lie algebra over $\mathbb{C}$,  $\mathbb{C}[t]$ be the polynomial algebra in $t$, and $n>1$ be an integer. Consider the Lie algebra 
$$L=\mathfrak{g}\otimes(t\mathbb{C}[t]/t^{2n+1}\mathbb{C}[t])=\oplus_{k=1}^{2n}(\mathfrak{g}\otimes t^k).$$
Note that $L$ is a nilpotent Lie algebra. 
As in (\ref{quo}), we have
$$\aligned L_1=L,\,\,\,& L_2\simeq \mathfrak{g}\otimes(t\mathbb{C}[t]/t^{2n-1}\mathbb{C}[t]),\,\,\, L_3\simeq \mathfrak{g}\otimes(t\mathbb{C}[t]/t^{2n-3}\mathbb{C}[t]),\\
 & \cdots, \,\,\,L_n\simeq \mathfrak{g}\otimes(t\mathbb{C}[t]/t^{3}\mathbb{C}[t]),\,\,\, L_{n+1}\simeq \mathfrak{g}\otimes(t\mathbb{C}[t]/t^{2}\mathbb{C}[t]).\endaligned$$ 

{\bf Step 1}. Find all commuting maps to $L_{n}$.

For any commuting linear map $f_n:L\to L_{n+1}$ which can be any linear map, let $g_n: L\to L_{n}$ be a commuting map such that $\tilde g_n=f_n$. We may assume that $g_n(\mathfrak{g}\otimes t)\subset \mathfrak{g}\otimes t$ up to a central map since $\mathfrak{g}\otimes t^2\subset Z_{L_{n}}(L')$. Let $g_n(x\otimes t)=h(x)\otimes t$. Then $h: \mathfrak{g}\to \mathfrak{g}$ is a commuting map on $\mathfrak{g}$ which has to be a scalar map by Corollary \ref{c2}. From this one can deduce that 
$g_n(\mathfrak{g}\otimes t^2\mathbb{C}[t])\subset \mathfrak{g}\otimes t^2$.

It is clear that any special commuting map is central. Thus every commuting map on $L_{n}$ is a sum of a scalar map (induced) and a central map.

{\bf Step 2}. Find all commuting maps on $L_{n-1}$.

For any commuting map $f_{n-1}:L\to L_{n}$ with
$$f_{n-1}(x\otimes t+y\otimes t^2)=ax\otimes t+(h_1(x)+h_2(y))\otimes t^2,$$
where $a\in\mathbb{C}$ and $ h_1, h_2$ are linear maps on $\mathfrak{g}$, let $g_{n-1}: L\to L_{n-1}$ be a commuting map such that $\tilde g_{n-1}=f_{n-1}$. Up to a central map  we may assume that 
$$g_{n-1}(x\otimes t)=ax\otimes t+h_1(x)\otimes t^2,\,\,\,
g_{n-1}(x\otimes t^2)=h_2(x)\otimes t^2.
$$
Using similar arguments as in Step 1 we deduce that  $h_1(x)=bx$ for some $b\in\mathbb{C}$ and $h_2(x)=ax$.
From this one can deduce that 
$g_{n-1}(\mathfrak{g}\otimes t^3\mathbb{C}[t])\subset \mathfrak{g}\otimes t^3$.

In this case it is also  clear that any special commuting map is central. Thus every commuting map  on $L_{n-1}$ is a sum of   a central map and a map $g_{n-1}$.

Continuing in this manner we deduce that, up to a central map,  every commuting map $f: L\to L$ is of the form
 $$f(x_k\otimes t^k)=x_k\otimes \sum_{j=k}^{2n}a_{j-k+1}t^j,\forall x_k\in \mathfrak{g},$$
 where $a_k\in\mathbb{C}.$ One can further see that $f$ lies in the centroid of $L$.
\qed
\end{example}

 We have   seen from the above examples that our methods cover a variety of results from the literature. Theoretically, one can repeatedly use Corollaries \ref{c2} and \ref{mt1'} to find all commuting linear maps on various Lie algebras.

\end{document}